\documentclass[10pt,reqno]{amsart}
\usepackage{amsmath, latexsym, amsfonts, amssymb,amsthm, amscd,epsfig,enumerate}
\usepackage{subfigure,color, float, caption}
\usepackage{tikz,colortbl}
\usetikzlibrary{calc}
\usetikzlibrary{decorations.pathreplacing}

\advance\hoffset -.75cm

\definecolor{refkey}{gray}{.75}
\definecolor{labelkey}{gray}{.75}

\newcommand{\N}{\mathbb N}

\newcommand{\pr}{\mathbb P}

\newtheorem{Theorem}{Theorem}[section]

\newtheorem{Corollary}[Theorem]{Corollary}

\newtheorem{Definition}[Theorem]{Definition}
\newtheorem{Example}[Theorem]{Example}

\title{Branching random walks with ageing}

\author[D.~Bertacchi]{Daniela Bertacchi}
\address{D.~Bertacchi, Dipartimento di Matematica e Applicazioni,
Universit\`a di Milano--Bicocca,
via Cozzi 53, 20125 Milano, Italy.}
\email{daniela.bertacchi\@@unimib.it}

\author[E.~Montanaro]{Elena Montanaro}
\address{E.~Montanaro, Dipartimento di Scienze Statistiche,
Sapienza Universit\`a di Roma,
Piazzale Aldo Moro 5, 00185 Roma, Italy.}
\email{elena.montanaro\@@uniroma1.it}

\author[F.~Zucca]{Fabio Zucca}
\address{F.~Zucca, Dipartimento di Matematica,
Politecnico di Milano,
Piazza Leonardo da Vinci 32, 20133 Milano, Italy.}
\email{fabio.zucca\@@polimi.it}

\date{}

\begin{document}

\begin{abstract}
Branching processes are models used to describe populations that reproduce and die over time. In the classical setting, an individual's reproductive capacity remains constant throughout its lifetime. However, in real-world situations, reproductive capacity typically undergoes ageing - that is, after reaching a peak, it decreases over time. In this work, we study the influence of ageing on the behaviour of the process and how
modifying its parameters, along with reproduction rates, affects the destiny of the process.
\end{abstract}

\maketitle


\section{Introduction}
\label{sec:intro}

Branching processes have been proposed, since the earlier studies (see for example \cite{cf:GW1875, cf:kendall60, cf:feller59}),
as models for populations where individuals randomly breed and die. These processes also have a long tradition (see \cite{cf:Griff73, cf:Clancy98}) in epidemic modelling.
The stochastic process represents, in these two situations, the total number of individuals of the population and the total number of infected individuals, respectively. 

Already in the 1970s the terminology \textit{branching random walks} appeared in the literature (see \cite{cf:ath-73,
 cf:joffe73, cf:asm76-1, cf:asm76-3}), denoting branching processes where individuals are endowed with a position and their position is
 chosen at random, given the position of their parent, according to the law of a random walk. 
 Branching random walks are a generalisation of branching processes and the addition of an infinite space set
 motivates further questions. For example, while branching processes have only two phases (indefinite growth or extinction),
 branching random walks add the distinction between local and global extinction and between local and global survival. The population may go extinct in any finite set, while still surviving as a whole (for details, see Definitions \ref{def:survival} and \ref{def:critical}).
 
 One of the most popular types of branching random walk is the continuous-time process where individuals have an 
 exponentially distributed lifespan and, during their lifetime, produce offspring at the arrival times of a homogeneous Poisson
 process with rate $\lambda>0$ (see Definition \ref{def:contBRW}).
  We refer to this process as the \textit{classical} continuous-time branching random walk.
 In the case of population modelling, $\lambda$ represents the reproductive capacity, 
 while in epidemic modelling, it represents transmissibility.
  Clearly, the larger $\lambda$ is, the more likely it is for the population to survive, and it is
 possible to define at least two critical parameters $\lambda_w\le\lambda_s$ (see Definition \ref{def:critical}).
 The identification of these two parameters has an applicative relevance, since if $\lambda<\lambda_w$, then the population goes
 globally extinct and if $\lambda<\lambda_s$, then the population goes  locally extinct.
 For results on the behaviour and the critical parameters of branching random walks, see for example
\cite{cf:Ligg1, cf:MachadoMenshikovPopov, cf:MP03, cf:Muller08-2, cf:PemStac1, cf:Stacey03}.

In real situations, an individual’s reproductive capacity, as well as the infectiousness of an individual affected by a transmissible disease, is not constant over time. Instead, it decreases with age or with time since infection. It is therefore natural to introduce ageing branching random walks, in which reproduction occurs at the arrival times of an inhomogeneous Poisson process with a decreasing rate function
 (see Definition \ref{def:ageBRW}).
We note that age-dependent branching processes have been investigated, for example, in \cite[Chapter IV]{cf:AthNey}
and \cite[Chapter VI]{cf:Harris63}, but, in those studies, age dependence means that the lifespan is not exponentially distributed; therefore, the probability of dying during a given time interval depends on age. In our framework, by contrast, the lifetime is exponentially distributed, as in the classical case, while the ability to reproduce decreases over time.
This definition is inspired by the concept of \textit{dissipation}, which has been studied in interacting particle systems, such as
spin models \cite{cf:daipra2, cf:daipra1} and contact processes \cite{cf:daipra3}.
In those papers dissipation affects the time between updates in the system: the update rate decreases over time, and has been shown to produce rhythmic oscillations in macroscopic quantities.

In the present paper, we focus on ageing branching random walks with a breeding rate of the form $\lambda e^{-\alpha t}$,
where $t$ represents the age of the parent. In particular, we prove results on the local and global critical parameters 
(Theorems \ref{th:locsurv-age}, \ref{th:suff-gl-ext-age}, \ref{th:glcrit}). In Section \ref{sec:crit-mod} we study how local modifications
of the rates affect the critical parameters, and compare with classical branching random walks with the same expected number of
offspring. We prove that in many cases (for example in transitive processes), survival or extinction of ageing branching random walks
 depends only on the  expected number of offspring. Nevertheless, the number of individuals, alive at time $t\ge0$, is a function that
 depends on the whole distribution of the process (not only on the expected number of offspring). We determine the behaviour, over time,
  of the  expected number of individuals in a branching process with ageing (Theorems \ref{th:ode}, \ref{th:compare}).
  This behavior differs from that of a process without ageing. This indicates that, to determine whether the process will survive or go extinct, it must be observed over a sufficiently long period of time. In Section  \ref{sec:concl}, we discuss our results and their implications for applications in species conservation, pest control, and epidemic control.


\section{Branching random walks: with or without ageing}
\label{sec:basic}

Given an at most countable  set of locations (or types) $X$, a \textit{branching random walk} (briefly BRW) is 
 a process $\{\eta_t\}_{t \in T}$,  where $\eta_t\in X^\N$.
 In the particular case where $X$ is a singleton, the process is named \textit{branching process}.
  For any given $t\in T$, $x\in X$, the value $\eta_t(x)$ represents the number of individuals 
in $x \in X$ (or of type $x$) alive at time $t$. Time can be discrete or continuous, i.e. $T\subseteq \N$ or
$T\subseteq[0,+\infty)$, respectively. The laws of the process, in continuous-time, discrete-time and with ageing, are
described in Sections \ref{subsec:continuous}, \ref{subsec:discrete} and \ref{subsec:ageing}, respectively.

We consider initial configurations with only one individual placed at a fixed site $x$ and we
denote by $\pr^{x}$ the law of the corresponding process. The evolution of the process
with more complex initial conditions can be obtained by superimposition, since reproductions and death of different individuals are
independent.
\begin{Definition}\label{def:survival} 
Given a process $\{\eta_t\}_{t\in T}$, with $\eta_t\in X^\N$, let $x\in X$ and
$A \subseteq X$.
\begin{enumerate}
 \item 
The process \textsl{survives} in $A$, with positive probability,  starting from $x$, if
$
{\mathbf{q}}(x,A)
:=1-\pr^{x}(\limsup_{t \to \infty} \sum_{y \in A} \eta_t(y)>0)<1.
$
If $A=\{x\}$, we say that the process \textsl{survives locally}, with positive probability, starting from $x$.
\item 
The process \textsl{goes extinct} almost surely in $A$, starting from $x$, if
$
{\mathbf{q}}(x,A) =1$.
\item
The process \textsl{survives globally}, with positive probability, starting from $x$, if
$
{\mathbf{q}}(x,X) 
<1$.
\item
The process \textsl{goes globally extinct}, almost surely, starting from $x$, if
$
{\mathbf{q}}(x,X) =1$.
\item
There is \textsl{strong survival} in $A$, starting from $x$, if
$ 
 {\mathbf{q}}(x,X)=
{\mathbf{q}}(x,A)<1
$ 
and \textsl{non-strong survival} in $A$, if $ {\mathbf{q}}(x,X)<{\mathbf{q}}(x,A)<1$.
\item
The process is in a \textsl{pure global survival phase}, starting from $x$, if
$ 
 {\mathbf{q}}(x,X)<{\mathbf{q}}(x,x)=1
$ 
(we write ${\mathbf{q}}(x,y)$ instead of ${\mathbf{q}}(x, \{y\})$ for all $x,y \in X$).
\end{enumerate}
\end{Definition}
Survival and extinction can be defined for single realisations of the process as well.
For example, if, given $x\in X$, $\{\eta_t(x,\omega)\}_{t\ge0}$, is strictly positive at arbitrarily large times, we say that 
for this $\omega$ there is local survival in $x$; if it is equal to zero for all sufficiently large $t$, we say that there is extinction in $x$.
We stress that in discussions about survival, the qualifier “with positive probability” is often implicitly understood. Similarly, in the case of extinction, the phrase “almost surely” is often omitted.

Strong survival in $A$ means that survival in $A$ happens with 
positive probability, and for almost all realisations, the process either survives in $A$ (hence globally) or goes extinct. More precisely,
there is strong survival in $y$, starting from $x$, if and only if the probability
of survival in $y$, starting from $x$, is positive and, conditioned on global survival, is equal to $1$.

On the other hand, negation of strong survival in $A$ means that either $ 
 {\mathbf{q}}(x,X)=1$ (almost sure global extinction) or $ 
 {\mathbf{q}}(x,X)<{\mathbf{q}}(x,A)
$ (positive probability of global survival, without survival in $A$).

The probabilities of extinction clearly depend on the law of the BRW.
In the classical case, the process is Markovian, since transitions depend only on the number and location of the individuals.
We remark that this is not the case when reproductive capacity is affected by age.
We first recall the definition of the classical BRW, in continuous or discrete time.

\subsection{Continuous-time Branching Random Walks}
\label{subsec:continuous}

Classical continuous-time BRWs are assigned by the choice of the space $X$ and of a  nonnegative matrix
$K=(k_{xy})_{x,y \in X}$.
Each individual has an exponentially distributed lifetime with parameter 1 (death occurs at rate 1). 
Fix $\lambda>0$: for each individual living in $x$, and each $y\in X$, such that $ k_{xy}>0$,
we associate a Poisson process with
parameter $\lambda k_{xy}$, which represents the breeding clock.
At each arrival time of this Poisson process, provided that the parent at $x$ is still alive, a newborn is placed at $y$.
The death and reproduction clocks associated with different individuals are independent.
We note that  the reproduction rate, $\lambda k_{xy}$, is the product of the basic rate $ k_{xy}$
and $\lambda$. We can view $\lambda$ as a knob that regulates the overall speed of reproduction (due to
abundance or lack of resources, the presence of a good climate, etc.), and $ k_{xy}$ as the local proclivity for reproduction.
We summarize the definition of continuous-time BRW in the following.
\begin{Definition}
\label{def:contBRW}
Given an at most countable set $X$ and  a  nonnegative matrix
$K=(k_{xy})_{x,y \in X}$, we denote by  $(X,K)$ the family of continuous-time BRWs, indexed by $\lambda>0$, where
each individual dies at rate 1 and individuals living in $x\in X$, produce offspring and send them to $y\in X$, at rate $\lambda k_{xy}$.
All breeding and death events, for different individuals, are independent. 
\end{Definition}

The law of the process is fully described by the set $X$, the matrix $K$ and the parameter $\lambda$.
Different values of $\lambda$ correspond to different processes in the same
family: $\lambda$ represents the speed of the single process.
With a slight abuse of notation, we often say that the family $(X,K)$ is a continuous-time BRW, rather than a family of BRWs.
We stress that, even if other processes that model independent births and deaths may be constructed in
continuous time, when we refer to \textit{continuous-time BRWs}, we assume that the process is a BRW $(X,K)$, constructed
via time-homogeneous Poisson processes.

The set $X$ is endowed with a graph structure $(X,E)$, induced by the matrix $K$
($E\subseteq X\times X$ being the set of edges).
Given $x,y\in X$, $(x,y)\in E$ if and only if $k_{xy}>0$. 
In other words, there is an edge from $x$ to $y$ whenever an individual at $x$ has
a positive probability of placing some offspring in $y$.
 We say that there is a path from $x$ to $y$, and we write $x \to y$, if it is possible to find a finite
sequence $ \{x_i\}_{i=0}^n$, for some $n \in \mathbb{N}$, such that $x_0 = x$, $x_n = y$ and $(x_i, x_{i+1}) \in E$
for all $i = 0, \ldots, n-1$ (observe that by definition we assume that there is always a path of length $0$ from $x$ to itself).
Whenever $x \to y$ and $y \to x$ we write $x \rightleftharpoons y$.
The equivalence relation $\rightleftharpoons $ induces a partition of $X$: the class $[x]$ of $x$ is called \emph{irreducible class} of $x$. If there is only one irreducible class, the process $(X,K)$ is said to be irreducible.

From the matrix of reproduction rates $K$ we also derive the 
\textit{first-moment matrix} $M=(m_{xy})_{x,y \in X}$, where each entry
$m_{xy}$ represents the expected number of children that an individual living
at $x$ sends to $y$ during its lifetime.
In continuous-time BRW, $m_{xy}=\lambda k_{xy}$. Note that $(x,y)\in E$ if and only if $m_{xy}>0$. 

We note that, in the definition of continuous-time BRWs $(X,K)$, the reproduction rates along the edge $(x,y)$ are given by
$\lambda k_{xy}$,
while it is implicitly assumed that the death rate is equal to 1 at each site.
 It is not difficult to see that the introduction
of a nonconstant death rate $\{d(x)\}_{x \in X}$ (in place of rate 1 at each site), does not represent a
significant generalizasion. Indeed,
 one can study the BRW with death rate 1 and
reproduction rates $\{\lambda k_{xy}/d(x)\}_{x,y \in X}$; the two processes have the same behaviours in
terms of survival and extinction (\cite[Remark 2.1]{cf:BZ14-SLS}). Hence, it suffices to study the behaviour of continuous-time BRWs with death rate equal to 1 at each site.

The process is monotone in $\lambda$: larger parameters imply faster breeding.
To be precise,
 given $\lambda_1<\lambda_2$,
we can couple the two corresponding processes, in a way such that the one with $\lambda_1$ 
has in all sites, and at all times,  almost surely,
 a number of individuals that does not exceed the corresponding number for the process with speed
$\lambda_2$. We note that the extinction probabilities ${\mathbf{q}}(\cdot,\cdot)$ depend on the choice of $\lambda$.
This leads to the definition of critical values of $\lambda$ (for a generalisation and further properties, see \cite{cf:BZ-critical26}).
\begin{Definition}\label{def:critical}
Let  $(X,K)$ be a continuous-time BRW. For any $x \in X$, there are two critical parameters: the \textsl{global} 
\textsl{survival critical parameter} $\lambda_w(x)$ and the  \textsl{local} 
 \textsl{survival critical parameter} $\lambda_s(x)$ (or, briefly, global and local critical parameters), defined as
\[ 
\begin{split}
\lambda_w(x)&:=
\inf \left\{\lambda>0\colon \,
 {\mathbf{q}}(x,X)<1
\right\}
,\\
\lambda_s(x)&:=
\inf\left\{\lambda>0\colon \,
\mathbf{q}(x,x)<1\right\}.
\end{split}
\] 
If  the
interval $(\lambda_w(x),\lambda_s(x))$ is not empty,
we say that there exists a \textsl{pure global survival phase} starting from $x$.
\end{Definition}
These values depend only on the irreducible class of $x$; in particular, they do not depend on $x$
if the BRW is irreducible. In the irreducible case we simply write $\lambda_w$ and $\lambda_s$.
Moreover, we say that the process is 
\textit{globally supercritical}, \textit{critical} or \textit{subcritical}
if $\lambda>\lambda_w$, $\lambda=\lambda_w$ or $\lambda<\lambda_w$ respectively.
An analogous definition of \textit{locally supercritical}, \textit{critical} or \textit{subcritical}
is given using $\lambda_s$ instead of $\lambda_w$.
No reasonable definition of a \textit{strong local survival critical parameter} is possible
(see \cite{cf:BZ14-SLS}).

The local survival critical parameter $\lambda_s(x)$ can be computed, thanks to a result that links it to the first-return generating function
$\Phi(x,x|\lambda)$.
We recall here the definition of the coefficients and of $\Phi$.
Let  
\begin{equation}\label{eq:defPhi}
\begin{split}
    \varphi^{(n)}_{xy}&:=\sum_{x_1,\ldots,x_{n-1} \in X \setminus\{y\}} k_{x x_1} k_{x_1 x_2} \cdots k_{x_{n-1} y}\\
    \Phi(x,y|\lambda)&:=\sum_{n =1}^\infty \varphi_{xy}^{(n)} \lambda^n.
\end{split}
\end{equation}
One can see the similarity between the coefficients $\varphi^{(n)}_{xy}$
and the first-arrival probabilities (or first-return, in case $x=y$) 
see~\cite{cf:Woess09} (Section 1.C).
In particular, $\lambda^n \varphi^{(n)}_{xy}$ is
the expected number of individuals alive at $y$ at time $n$,
when the initial state is just one individual at $x$ and the
process behaves like the original BRW except that every individual reaching
$y$ at any time $i <n$ is immediately killed (before breeding).
The following characterisation holds (see~\cite{cf:BZ2}):
\begin{equation}
\label{eq:lambdas1}
\lambda_s(x)=
\max\{ \lambda 
\in {\mathbb R}:\Phi(x,x|\lambda)\leq 1\}=
\sup 
\{ \lambda 
\in {\mathbb R}:\Phi(x,x|\lambda)< 1\}.
\end{equation}
Unfortunately, the characterisation of the global survival critical parameter is  far less explicit, and we are not going to use it here.
We refer the interested reader to \cite{cf:BZ2}.

\subsection{Discrete-time Branching Random Walks}
\label{subsec:discrete}

In discrete-time BRWS, individuals  live one unit of time; at their death, they are replaced by their offspring.
The number of children and their locations depend only on the location of the parent.
In order to give a formal description, 
\begin{Definition}
\label{def:discreteBRW}
Let $X$ be an at most countable set and let 
 $S_X := \{f \colon X \to \N : \sum_y f(y) < \infty \}$.
 Consider a family $\mu = \{\mu_x\}_{x \in X}$ of probability measures on the (countable) measurable space $(S_X, 2^{S_X})$.
 Denote by $(X, \mu)$ the discrete-time process in which each individual lives one unit of time and at their death, they are
 replaced by $f(y)$ individuals at $y$, for all $y \in X$.
The function $f$ is chosen according to $\mu_x$, where $x$ is the location of the parent.
All breeding and death events, for different individuals, are independent. 
\end{Definition}


In the discrete-time case,  the  first-moment matrix $M = (m_{xy})_{x,y \in X}$, has entries $ m_{xy} := \sum_{f \in S_X} f(y) \mu_x(f)$.
 For simplicity, we require $ \sup_{x \in X}{\sum_{y \in X}{m_{xy}} < + \infty}$.
 The graph structure $(X,E)$ is analogous to the continuous-time case: $(x,y)\in E$ if and only if
$m_{xy}>0$.
  In order to avoid trivial situations where individuals have one offspring almost surely, we also
 assume that in every equivalence class (with respect to $\rightleftharpoons$) there is at least one vertex
where an individual can have, with positive probability, a number of children different from $1$.

We observe that, in contrast with the continuous-time case, here the number of children and their locations are not 
necessarily independent variables. For example, for any $X$ with at least three distinct sites $y,w,z$, let $\delta_{i}(j)=1$ if $i=j$, $\delta_{i}(j)=0$
if $i\neq j$. Choose
$\mu_x$  such that $\mu_x(f_1)=1/3$ and 
$\mu_x(f_2)=2/3$, with $f_1=2\delta_{y}+\delta_{w}$ and $f_2=\delta_{z}$.
If the number of children and their locations are  independent variables, then we say that the BRW is a
BRW \textit{with independent diffusion}. A BRW with independent diffusion is fully described by the laws $\rho_x$,
which regulate the total number of children of an individual that lives in $x\in X$, and the transition matrix $P$, defined
as follows:
\begin{equation*}
\begin{split}
 \rho_x(n) &:= \mu_x(\{f\colon \sum_y f(y)=n\}),\\
 p(x, y) & := m_{xy}/\sum_{n \geq 0} n \rho_x(n).
\end{split}
\end{equation*}
Note that the expected value of $\rho$, $\sum_{n \geq 0} n \rho_x(n)$, coincides with $\sum_{y\in X} m_{xy}$.
For a BRW with independent diffusion we have:
\begin{equation}\label{eq:indepdif}
	\mu_x(f) = \rho_x \left( \sum_{y} f(y) \right)
	\frac{ ( \sum_{y} f(y))! }{ \prod_{y} f(y)! }
	\cdot \prod_{y} p(x, y)^{f(y)}, 
	\quad \forall f \in S_X.
\end{equation}

\subsection{The discrete-time counterpart of a continuous-time BRW}
\label{subsec:discretecounterp}

For any continuous-time BRW, we associate
a discrete-time BRW, which we call its \textit{discrete-time counterpart}.
Time in the discrete-time counterpart represents generation: 
at time $n+1$ the only individuals alive are those who
are offspring of individuals of the $n$-th generation.
 If we define
\begin{equation*}
\rho_x(i)=\frac{1}{1+\lambda k(x)} \left ( \frac{\lambda k(x)}{1+\lambda k(x)} \right )^i, \qquad
p(x,y)=\frac{k_{xy}}{k(x)}, \qquad k(x):=\sum_{y \in X} k_{xy}, 
\end{equation*}
then it is easy to show that
$\mu_x$ satisfies Equation \eqref{eq:indepdif}.

Note that the two processes do not share the same time clock and some individuals of the $n$-th
generation may be born after individuals of the $m$-th generation, even if $n>m$.
However, the total number of individuals ever born at each site is the same for both processes.
This means that the continuous-time BRW survives (in $A$, strongly in $A$ or globally)  with the same probability as its discrete-time counterpart.

\subsection{BRW with ageing}
\label{subsec:ageing}

In the classical continuous-time BRW, each individual maintains its reproductive capacity throughout
their entire life (the Poisson clock that regulates breeding is time-homogeneous, i.e., it has a constant parameter $\lambda$).
In real life such a capacity typically decays with age. It is therefore natural to extend the definition of BRW to 
the case where breeding occurs at the arrival times of an inhomogeneous Poisson
process.

\begin{Definition}
\label{def:ageBRW}
Given an at most countable set $X$ and  a family
$
\mathcal{R} := \{r_{xy}\}_{x,y \in X}
$, where each $r_{xy}\in L^1_{loc}([0,+\infty))$, we  denote by
$(X, \mathcal{R})$
 the family of  BRWs, indexed by $\lambda>0$, 
 defined as follows.
Each individual dies at rate 1 and an individual living in $x\in X$, produces offspring and sends them to $y\in X$, 
at the arrival times of a Poisson point process of intensity
$\lambda r_{xy}(t-\bar t)$, where $\bar t\ge0$ is the time of birth of the parent.
All breeding and death events, for different individuals, are independent. 
These processes are called \textsl{BRWs with ageing} or  \textsl{ageing BRWs}.
\end{Definition}
%

In an ageing BRW, the average number of children placed in $y$ by an individual at $x$ during its lifetime is
\begin{equation}\label{eq:1mom-ageing}
m_{xy} = \lambda \int_0^{\infty} r_{xy}(s)\,\exp(-s)\,ds. 
\end{equation}
As in the discrete-time case, 
we require $\sup_{x \in X} \sum_{y \in X}m_{xy}<+\infty$.
 The graph structure $(X,E)$ is again defined in terms of the first moment matrix: $(x,y)\in E$ if and only if
$m_{xy}>0$.

We note that continuous-time BRWs are a particular case of ageing BRWs, where $r_{xy}(s)=k_{xy}$,
 for all $s\in [0,+\infty)$.
 As for classical continuous-time BRWs, the family is monotone in $\lambda$. Therefore, we are able to define the critical parameters
$\lambda_w(\cdot)$ and $\lambda_s(\cdot)$ of \((X, \mathcal{R})\), as in Definition \ref{def:critical}.
We call the family \emph{ageing BRW} or \emph{BRW with ageing}.
Clearly, breeding is subject to ageing, in common sense, when functions $r_{xy}$ are eventually non-increasing.
Nevertheless, we use this terminology for all choices of the family $\mathcal R$.

For any given $\lambda>0$, the process $\{\eta_t\}_{t\ge0}$, where $\eta_t(x)$ is the number of individuals alive in $x$ at time $t$,
 is not a Markov process, unless the intensities are constant functions, i.e. unless it coincides with a continuous-time classical BRW.
The process becomes Markovian once we keep track of the age of each individual, that is, we consider the process 
$\{\eta_t, A_{n,t}\}_{t\ge0, n\in\mathbb N}$, where $A_{n,t}(x)$ represents the age, at time $t$, of the $n$-th individual born in the site $x$.

We observe that a choice of more general lifetimes does not provide a meaningful generalisation, as long
as we are only interested in the first moment matrix.
Indeed, if each individual at $x$ has a random lifetime with cumulative distribution function \(T_x\),
then by writing $\widetilde r_{xy}(s)=r_{xy}(s)\,(1 - T_x(s))\exp(s)$, we get
\begin{equation*}
\begin{split}
m_{xy} &= \lambda \int_0^{\infty} r_{xy}(s)\,(1 - T_x(s))\,ds\\
& =\lambda \int_0^{\infty} \widetilde r_{xy}(s)\,\exp(-s)\,ds.
\end{split}
\end{equation*}

For any ageing BRW, we associate its discrete-time counterpart, where time represents generation.
We may compute the measure $\mu_x$, which rules the number and locations of children of a parent living at $x$
\begin{equation}\label{eq:discrete-ctp-ageing}
\mu_x(f) = \int_0^{\infty}  \prod_{y \in X} \left( \frac{\exp(-\lambda \int_0^{t} {r_{xy}(s) ds} )(\lambda \int_0^{t} {r_{xy}(s) ds})^{f(y)}}{f(y)!} \right) \exp(-t)\, dt.
 \end{equation}

We note that the asymptotic behaviour of a BRW with ageing and of its discrete-time counterpart coincide: the two processes
share the same extinction probabilities and they also share the same first moment matrix.
Moreover, even if the ageing BRW is not markovian, its discrete-time counterpart is, since it takes into account the children that an
individual produces, during their whole life.

\section{Survival and extinction of a BRW with ageing}
\label{sec:surv}

Since BRWs with ageing share the first moment matrix with their discrete-time counterpart, whenever survival depends only
on the first moments of the process, we are able to carry the results for the latter, to the former.
In particular, we are able to fully characterise local survival.
\begin{Theorem}\label{th:locsurv-age}
	Let $(X,\mathcal{R})$ be an ageing BRW.
	Let $K= (k_{xy})_{x,y \in X}$ be the matrix defined by
	\begin{equation*}
	k_{xy}= \int_0^{\infty} r_{xy}(s)\,\exp(-s)\,ds.
	\end{equation*}
	Denote by $k^{(n)}_{xy}$ the entries of its $n$-th power.
	 Fix $\lambda>0$ and $x\in X$. The process survives locally at $x$ if and only if 
	 \begin{equation*}
	 \lambda  \limsup_{n \to +\infty} \sqrt[n]{k^{(n)}_{xx}} > 1.
	 \end{equation*}
	Equivalently, $\lambda_s(x)=1/  \limsup_{n \to +\infty} \sqrt[n]{k^{(n)}_{xx}}$ and
	\begin{equation}
\label{eq:lambdas2}
\lambda_s(x)=
\sup 
\{ \lambda 
\in {\mathbb R}:\Phi(x,x|\lambda)< 1\}.
\end{equation}
		\end{Theorem}
\begin{proof}
For any fixed $\lambda$, the ageing BRW has the same extinction probabilities as its discrete-time counterpart, which is given by the
reproduction measures in Equation \eqref{eq:discrete-ctp-ageing}, and the same first-moment matrix $M=(m_{xy})$, defined by
Equation \eqref{eq:1mom-ageing}.
The claim is a direct consequence of \cite[Theorem 2.4]{cf:BZ17},
which links local survival with the first-moment matrix and allows
us to use the characterisation in Equation \eqref{eq:lambdas2}.
\end{proof}

In general, it is more difficult to characterise global survival. The following result provides a sufficient condition for global extinction, 
which is similar to the one for local extinction, stated in Theorem \ref{th:locsurv-age}.
\begin{Theorem}\label{th:suff-gl-ext-age}
Let $(X,\mathcal{R})$ be an ageing BRW.
	Let  $k^{(n)}_{xy}$ be as in Theorem  \ref{th:locsurv-age}.
	 Given a fixed $x\in X$,
	  $\lambda_w(x)\ge 1/ \liminf_{n \to +\infty} \sqrt[n]{\sum_{y\in X}k^{(n)}_{xy}}$.
	  Equivalently, if $\lambda<1/ \liminf_{n \to +\infty} \sqrt[n]{\sum_{y\in X}k^{(n)}_{xy}}$, then the process goes  extinct almost surely.
\end{Theorem}
\begin{proof}
Using the discrete-time counterpart of the BRW, the claim follows from \cite[Theorem 2.5.2]{cf:BZ17}.
\end{proof}
If the discrete-time counterpart has some symmetries, such as transitivity or quasi-transitivity, then it is possible to characterise the global critical parameters.
Quasi-transitivity corresponds to the fact that there is
a finite subset $X_0\subseteq X$ such that for all $x\in X$ there exists an automorphism $\gamma:X\to X$, with $\gamma(x)\in X_0$, and 
the reproduction laws are $\gamma$-invariant. If $|X_0|=1$, then the process is said to be transitive.
A larger class of BRWs, which includes the transitive and quasi-transitive case, is the class of $\mathcal F$-BRW
(see, for example, \cite[Section 2.4]{cf:BZ17}). Roughly speaking, from the viewpoint of their reproduction laws,
 these processes have a finite number of ``types of neighbourhoods''. 
\begin{Theorem}\label{th:glcrit}
Let $(X,\mathcal{R})$ be an ageing BRW and suppose that its discrete-time counterpart is an $\mathcal F$-BRW.
Fix $x\in X$, $\lambda>0$. 
There is global survival, starting from $x$, if and only if $\lambda > 1/\liminf_{n \to +\infty} \sqrt[n]{\sum_{y\in X}k^{(n)}_{xy}}$.
\end{Theorem}
\begin{proof}
The claim follows from \cite[Theorem 2.5.3]{cf:BZ17}, which states that the same holds for the discrete-time counterpart of the process.
\end{proof}

\section{Critical parameters under local modifications}
\label{sec:crit-mod}

From now on, unless otherwise stated, all processes in the paper will be irreducible.
In BRWs, an interesting topic is how the behaviour of the process changes after local modifications of the reproduction laws.
In continuous-time BRWs, local changes may affect the value of critical parameters.
This phenomenon has been investigated in detail in \cite{cf:BZmathematics}, and extended to more general families of BRWs in \cite{cf:BZ-critical26}. 
To be precise, we formally define local modifications of ageing BRWs. Note that the definition also covers the case of continuous-time
BRWs, since these processes are particular cases of ageing BRWs, where the reproduction rates are constant.
\begin{Definition}
Let $(X,\mathcal R)$ and $(X,\mathcal R^*)$ be two ageing BRWs, defined on the same space $X$.
If the following set is finite
\[
\{x \in X \colon \exists y \in X, \, r_{xy}\neq r^*_{xy}\},
\]
 then we say that $(X,\mathcal R)$  is a \textsl{local modification of} $(X,\mathcal R^*)$  (and vice versa).
\end{Definition}
Let $(X,K)$ be an irreducible continuous-time BRW,
 with critical parameters $\lambda_w$ and $\lambda_s$, and consider  a local
 modification $(X,K^*)$.
 Let $\lambda_w^*$ and $\lambda_s^*$ be
the critical parameters of $(X,K^*)$.
It has been proven (see \cite[Corollary 2, Corollary 3, Figure 1]{cf:BZmathematics}) that if  $(X,K)$ has a pure global survival phase (i.e. $\lambda_w<\lambda_s$), then it is not possible to increase $\lambda_w$ through local 
modifications of the parameters, that is, $\lambda_w^*\le\lambda_w$. Moreover, if  $\lambda_w\neq \lambda_w^*$, this means
that $\lambda_w^*=\lambda_s^*<\lambda_w$.
Equivalently, if $\lambda_w^*<\lambda_s^*$, then $\lambda_w=\lambda_w^*$.

In some cases, these results can also be exploited to retrieve the value of the global critical parameter.
Indeed, suppose that $(X,K)$ is an $\mathcal F$-BRW (recall that examples are quasi-transitive BRWs)
and $\lambda_w<\lambda_s$. A local modification $(X,K^*)$ is highly likely to lack the property of being an $\mathcal F$-BRW.
In this case, we do not have an explicit expression for $\lambda_w^*$,
but we have the lower bound in Theorem \ref{th:suff-gl-ext-age}, 
and we know that $\lambda_w^*\le\lambda_w$.
On the other hand, the value of $\lambda_s^*$ is known, by
Theorem \ref{th:locsurv-age}. If we are able to prove that $\lambda_w^*<\lambda_w$, or that $\lambda_s^*\le\lambda_w$,
then we know that $\lambda_w^*=\lambda_s^*$.

The goal of this section is twofold. On the one hand, we study what happens to the critical parameters
of an ageing BRW after 
local modifications, using
Theorem \ref{cor:maximality}. On the other hand, we compare an ageing BRW with the classical continuous-time
BRW that shares the same first-moment matrix. In particular, we investigate the presence or absence of the pure global survival phase.

We note that the results in \cite{cf:BZmathematics}, which deal with the critical parameters of continuous-time BRWs under local
changes, cannot be applied directly to ageing BRWs, since, in general, an ageing BRW and a continuous-time BRW, even if they have the
same first-moment matrix, may have different behaviours.
Instead, we make use of the following result that applies to the family of ageing BRWs.
We omit the proof, since it is a direct application of the theorems in \cite{cf:BZ-critical26}.

 \begin{Theorem}\label{cor:maximality}
Let 
$(X,\mathcal R)$ and $(X,\mathcal R^*)$  be two ageing BRWs.
Denote by $\lambda_w$,
 $\lambda_s$, $\lambda_w^*$,
 $\lambda_s^*$, the critical parameters of the two processes.
Suppose that $\lambda_w < \lambda_s$ and that  $(X,\mathcal R)$ is a local modification of $(X,\mathcal R^*)$.
%
Then we have the following.
\begin{enumerate}
    \item $\lambda_w^{*} \le \lambda_w$.
    \item If $\lambda_w^* < \lambda_s^*$, then $\lambda_w^* = \lambda_w$.
    \item If $\lambda_s^* \le \lambda_w$, then $\lambda_w^* = \lambda_s^*\le\lambda_w$. If $\lambda_s^* > \lambda_w$, then $\lambda_w^* = \lambda_w<\lambda_s^*$.
\end{enumerate}
\end{Theorem}

In the following examples, we consider irreducible processes; therefore,
$\lambda_s$ and $\lambda_w$ do not depend on the vertices. In these examples $X=\mathbb{T}_d$, the homogeneous tree of degree $d>3$, and we write $x\sim y$ for vertices
$x,y\in \mathbb T_d$ that are adjacent. It is well known that rapidly growing trees, such as homogeneous trees, are the ideal setting
where populations may survive on the graph, although they eventually vacate any finite portion of the graph itself
(for early papers on BRWs on trees, see \cite{cf:Ligg1, cf:Ligg2, cf:MadrasSchi, cf:MP03, cf:PemStac1}).
We formally state this fact by recalling the example of the continuous-time BRW, which is adapted to the graph structure, i.e. the breeding 
may occur only along the edges of the graph, and homogeneous in the sense that it spreads equally along any edge 
(Example \ref{ex:homtree}). 
 The critical parameters of this BRW are well-known. They are evaluated, for example, in \cite[Example~4.2]{cf:BZ14-SLS}.
 We provide a proof for the  sake of completeness. This proof, for the computation of 
 $\lambda_s$, is different from the one in \cite{cf:BZ14-SLS}, and exploits the first-return generating function $\Phi$ and its connection
 with $\lambda_s$ (see Equations \eqref{eq:defPhi} and \eqref{eq:lambdas1}).
 
\begin{Example}\label{ex:homtree} 
	Fix $k>0$ and consider the continuous-time BRW $(X,K)$ on the homogeneous tree $\mathbb{T}_d$ of degree $d\ge3$, where $K$ is
	defined as $k$ times the adjacency matrix of $\mathbb T_d$.
	   The critical parameters satisfy
	\[
	\lambda_w = \frac{1}{kd} < \frac{1}{2k\sqrt{d-1}} = \lambda_s.
	\]
	Hence, 
	there exists a \emph{pure global survival phase}. 
\end{Example}
 \begin{proof}
The total number of individuals alive at time $t$ coincides with the same quantity for a continuous-time branching process
with rate $\lambda kd$ (equivalently, this is the BRW on the singleton  $X=\{x\}$, with $K=(k_{xx})$, $k_{xx}=kd$).
Since a branching process survives if and only if the expectation of the offspring distribution is larger than 1, $\lambda_w=1/kd$.

In order to compute $\lambda_s$, we make use of the characterisation in Equation \eqref{eq:lambdas1}. Fix a vertex $o$ and call it root,
and name $x_1$ one of the neighbours of $o$. From rotational invariance and absence of cycles, we get
\[
\begin{split}
\Phi_{\mathbb{T}_d}(o,o|\lambda)& = d\lambda k \cdot\Phi_{\mathbb{T}_d}(x_1,o|\lambda),\\
\Phi_{\mathbb{T}_d}(x_1,o|\lambda)&= \lambda k + (d-1)\lambda k\cdot \Phi^2_{\mathbb{T}_d}(x_1,o|\lambda).
\end{split}
\]
By solving the quadratic equation, we obtain
\begin{equation*}
\Phi_{\mathbb{T}_d}(o,o|\lambda) =  d \frac{1-\sqrt{1-4(d-1)\lambda^2 k^2}}{2(d-1)},
\end{equation*}
whence, from Equation \eqref{eq:lambdas1},  $\lambda_s =1/(2k\sqrt{d-1}) $.
 \end{proof}
 
 In the following examples, we study how local modifications of this process may or may not affect the critical parameters.
The next example is a local modification of Example \ref{ex:homtree} and generalises \cite[Example 1]{cf:BZmathematics}, where the case $k=1$ was studied.
We observe that the addition of a fast local reproduction can destroy the pure global survival phase: when $k_{oo}$ is small,
 the critical parameters coincide with those in Example \ref{ex:homtree}; for intermediate values the global critical parameter
 is unchanged, while the local one decreases in $k_{oo}$; for large values of $k_{oo}$ they coincide.
\begin{Example}\label{ex:treeloop}
Fix a vertex $o \in \mathbb{T}_d$ and call it the root of the tree.
Consider the continuous-time BRW $(X,K^*)$, with $K^*=(k^*_{xy})$ given by
\[
  k^*_{xy} = 
  \begin{cases}
k & \text{if } x \sim y,\\
	k_{oo}  & \text{if } x = y = o,
   \end{cases}
   \]
   where $k$ and $k_{oo}$ are two nonnegative parameters.
   Let $\lambda_w^*$ and $ \lambda_s^*$ be the two critical parameters of the process. Then we have the following explicit expressions.
   \begin{enumerate}
       \item 
       If $k_{oo} \le \frac{k(d-2)} {\sqrt{d-1}}$, then $\lambda_w^* = \frac{1}{kd} <\lambda_s^*= \frac{1}{2k\sqrt{d-1}} $.
\item  If $\frac{k(d-2)} {\sqrt{d-1}}< k_{oo} \le  \frac{kd(d-2)}{d-1}$, then
 $\lambda_w^* = \frac{1}{kd} <\lambda_s^*=\frac{(d-2)k_{oo}+d \sqrt{{k_{oo}}^2+4k^2}}{2 \big((dk)^2+(d-1){k^2_{oo}} \big)}$.
 \item  If $ \frac{kd(d-2)}{d-1}\le k_{oo}$, then
 $\lambda_w^* = \lambda_s^*=\frac{(d-2)k_{oo}+d \sqrt{{k_{oo}}^2+4k^2}}{2 \big((dk)^2+(d-1){k^2_{oo}} \big)}$.
   \end{enumerate}
 In particular, if  $k_{oo} < kd(d-2)/(d-1)$, then $\lambda_w^* < \lambda_s^*$; if $k_{oo} \geq kd(d-2)/(d-1)$, then
$\lambda_w^*= \lambda_s^*$.
 \end{Example}
 \begin{proof}
   Let $\Phi^*$ be the first-return generating function of the BRW  and let $\Phi_{\mathbb{T}_d}$ be the corresponding one of the BRW on $\mathbb{T}_d$, described in Example \ref{ex:homtree}. It is easy to see that
	\[
	\Phi^*(o,o|\lambda) = \Phi_{\mathbb{T}_d}(o,o|\lambda) + k_{oo}\lambda 
	= d \frac{1-\sqrt{1-4(d-1)\lambda^2 k^2}}{2(d-1)} + k_{oo}\lambda.
	\]
From the above expression of $\Phi^*$ and Equation \eqref{eq:lambdas1}, we derive the values of $\lambda_s^*$, for different values of $k_{oo}$. 
Since this BRW is a modification of the one in Example \ref{ex:homtree}, if we use the notation $\lambda_w$ and $\lambda_s$ for the
critical parameters  in that example,
from \cite[Corollary 2]{cf:BZmathematics}, we know that 
if $\lambda_s^*>\lambda_w$, then $\lambda_w^*=\lambda_w$.
Moreover, if $\lambda_s^*\le\lambda_w$, then
$\lambda_w^*=\lambda_s^*$.
\end{proof}

Now we add ageing to the process in Example \ref{ex:treeloop}.
In the following examples, ageing is modeled by an exponential decay in the time of the birth rate.
In Example \ref{ex:agelooptree}, each location is affected by age and ageing affects the root and the other vertices of the tree differently.
	
\begin{Example}\label{ex:agelooptree}
Let $o$ be a fixed vertex in the homogeneous tree $\mathbb T_d$, and fix positive real numbers $k, k_{oo},\alpha, \alpha_0$.
Let $(\mathbb T_d,\mathcal R^*)$ be the ageing BRW, with the following reproduction rate functions
\[
r^*_{xy} = 
\begin{cases} 
	    {k_{oo}}\cdot\exp(-\alpha_o t) & \text{if } x = y = o, \\
	   {k}\cdot\exp(-\alpha_o t) & \text{if } x = o, y \sim o, \\
   k\cdot\exp(-\alpha t)  & \text{if } x \neq o, y \sim x.
\end{cases}
\]
Let $\lambda_w^*$ and $\lambda_s^*$ be its global and local survival critical parameters, respectively.
Let $\lambda_w$ and $\lambda_s$ be the global and local survival critical parameters of the process
 with 
$k_{oo}=0$ and $\alpha_o=\alpha$.
There exist $k_1$ and $k_2$, depending on $d,k,\alpha,\alpha_o$, such that $k_2>\max(0,k_1)$, is increasing in $\alpha_0$ and the following
hold.
\begin{enumerate}
\item
If $\frac{\alpha_0+1}{\alpha+1}>\frac d{2(d-1)}$, then $k_1>0$ and for
$k_{oo}\in[0,k_1)$, $\lambda_w=\lambda_w^*<\lambda_s^*=\lambda_s$.
\item
If $\frac{\alpha_0+1}{\alpha+1}\le\frac d{2(d-1)}$ and $k_{oo}\in(0,k_2)$, then
$\lambda_w=\lambda_w^*<\lambda_s^*<\lambda_s$.
\item
If $k_{oo}\ge k_2$, then $\lambda_w^*=\lambda_s^*<\lambda_w$.
\end{enumerate}
\end{Example}
\begin{proof}
Let us denote by $(\mathbb T_d,\mathcal R)$, the 
ageing BRW, with reproduction rate functions
\[
r_{xy} = 
   k\cdot\exp(-\alpha t), \qquad \text{if } x  \sim y.
\]
Note that  $(\mathbb T_d,\mathcal R)$ is a particular case of $(\mathbb T_d,\mathcal R^*)$, 
where $\alpha_o=\alpha$ and $k_{oo}=0$.
First, we determine its  critical parameters, 
$\lambda_w$ and $\lambda_s$.
The discrete-time counterpart of $(\mathbb T_d,\mathcal R)$ shares the first-moment matrix with
 the discrete-time counterpart of the continuous-time BRW $(\mathbb T_d,K)$,
$K=(k_{xy})_{x,y\in\mathbb T_d}$, $k_{xy}=k/(\alpha+1)$ if $x\sim y$, $k_{xy}=0$ otherwise.
Since the discrete-time counterparts of $(\mathbb T_d,\mathcal R)$ and $(\mathbb T_d,K)$ are transitive BRWs
(although the first is an ageing BRW and the second is a continuous-time BRW), their critical parameters depend 
only on the first-moment matrix and thus they also share the global and local critical parameters.
Hence, as shown in Example \ref{ex:homtree}, 
$\lambda_w = (\alpha +1)/(kd)$, $  \lambda_s=(\alpha+1)/(2k\sqrt{d-1})$. Clearly, for
all $d\ge3$, the process has a pure global survival phase.

Now we note that $(\mathbb T_d,\mathcal R^*)$, when $\alpha_o\neq\alpha$ and $k_{oo}\neq0$, is a local modification of $(\mathbb T_d,\mathcal R)$.
If we fix $\alpha$, $\alpha_o$, $k$ and $k_{oo}$, then the families of processes 
 $(\mathbb T_d,\mathcal R^*)$ and  $(\mathbb T_d,\mathcal R)$, indexed by $\lambda$, are ordered
 families of BRWs.
Therefore, we are able to apply Theorem \ref{cor:maximality}, and claim that 
$\lambda_w^*\le \lambda_w=(\alpha +1)/(kd)$, and that
$\lambda_s^*\le\lambda_w$ if and only if
  $\lambda_w^*=\lambda_s^*$.

In order to compute $\lambda_s^*$, we note that we can use the discrete-time counterpart of
 $(\mathbb T_d,\mathcal R^*)$, which coincides with the one of the 
 continuous-time BRW that shares the same expected number of offspring at each site.
 Since, by Theorem \ref{th:locsurv-age}, the value of the local survival critical parameter depends only on the first-moment matrix, 
 we are left with the task of studying the 
 continuous-time BRW $(\mathbb T_d,K^*)$, with $K^*=(k^*_{xy})_{x,y\in\mathbb T_d}$, 
 \[
{k}^*_{xy} =
\begin{cases} 
	    \frac{k_{oo}}{\alpha_o +1} & \text{if } x = y = o, \\
	  \frac{k}{\alpha_o +1} & \text{if } x = o, y \sim x, \\
    \frac{k}{\alpha +1} & \text{if } x \neq o, y \sim x.
\end{cases}
\]
The values of ${k}^*_{xy}$ are computed by evaluating the expected number of children:
\[
{k}^*_{xy} = \int_0^\infty r^*_{xy} \exp(-\alpha_x t) \exp(-t)dt= \frac{k_{xy}}{\alpha_x+1},
\]
where we put $\alpha_x:=\alpha$, when $x\neq o$, and $k_{xy}=k$, when $x\neq o$, $x\sim y$.

We proceed to evaluate the first-return generating function $\Phi^*$, by comparison with
the generating function $ \Phi_{\mathbb{T}_d}$, calculated in Example \ref{ex:homtree}.
Note that the first-return paths can be partitioned according to the first
step. If the first step is along the loop $(o,o)$, then the expected number
of children along this path is ${\lambda k_{oo}}/(\alpha_o +1)$; if the
first step is from $o$ to $x$ ($x\sim o$), then the path is entirely
contained in $\mathbb T_d$ and the expected number of first-return
children is the same and in $(\mathbb T_d,K)$, multiplied by
$(\alpha +1)/(\alpha_o+1)$ (since the contribution of the expected number of children of the first step is $\lambda k/(1+\alpha_o)$ instead of 
$k\lambda/(1+\alpha)$).
Therefore, we have
\begin{equation}\label{eq:phi*}
\begin{split}
 \Phi^*(o,o| \lambda) & = \frac{\alpha +1}{\alpha_o+1}\cdot \Phi_{\mathbb{T}_d}(o,o| \lambda k / (\alpha +1)) + \frac{\lambda k_{oo}}{\alpha_o +1} 
\\
& = \frac{\alpha +1}{\alpha_o +1} \cdot d \frac{1-\sqrt{1-4(d-1)(\lambda k / (\alpha +1))^2}}{2(d-1)}+ \frac{\lambda k_{oo}}{\alpha_o +1}. 
\end{split}
\end{equation}
Since the square root decreases in $\lambda$, and is real if and only if
 $\lambda\le \frac{\alpha+1}{2k\sqrt{d-1}}$,
then 
\begin{align*}
 \Phi^*(o,o| \lambda) & \le \Phi^*\left(o,o\,\Big| \frac{\alpha+1}{2k\sqrt{d-1}}\right)\\ 
& = \frac{\alpha +1}{\alpha_o +1}  \frac d{2(d-1)}+ \frac{k_{oo}}{\alpha_o +1} \frac{\alpha+1}{2k\sqrt{d-1}}. 
\end{align*}
The last quantity is not larger than 1 if $ k_{oo} \le k_1:=\Big( \frac{\alpha_o+1}{\alpha+1}- \frac{d}{2(d-1)} \Big) 2k \sqrt{d-1}$.
This implies that,
if  $k_{oo} \le k_1$,  then $\lambda_s^*=\lambda_s=\frac{\alpha+1}{2k\sqrt{d-1}}$ and $\lambda_w^*=\lambda_w= \frac{\alpha +1}{kd}$.
Note that $k_1>0$ if and only if $\frac{\alpha_0+1}{\alpha+1}>\frac d{2(d-1)}$, otherwise it
is not possible that $k_{oo}>k_1$.

On the other hand, since
\begin{align*}
 \Phi^*\left(o,o\Big| \frac{\alpha +1}{kd}\right) & = 
 \frac{\alpha +1}{\alpha_o +1} \cdot d \frac{1-\sqrt{1-4(d-1)/d^2}}{2(d-1)}+ \frac{k_{oo}}{\alpha_o +1} 
 \frac{\alpha +1}{kd}\\
 & =
  \frac{\alpha +1}{\alpha_o +1}  \frac{d(3-d)}{2(d-1)}+ \frac{k_{oo}}{kd}
   \frac{\alpha +1}{\alpha_o +1} ,
\end{align*}
 equals 1 if $k_{oo}= k_2:=kd\left(\frac{\alpha_o +1}{\alpha +1}+\frac {d(d-3)}{2(d-1)}\right)$, then
$\lambda_s^*=\lambda_w^*$ if $k_{oo}\ge k_2$.
%
\end{proof}
We remark that, in order to study the behaviour and the critical parameters of the general process in 
Example \ref{ex:agelooptree}, we cannot assume that the critical parameters coincide with the ones of
the continuous-time BRW with the same first-moment matrix. This is only true for the local critical parameter, but
in general it is not true for the global parameter, which depends not only on the first moments but also
on the offspring distribution. 

Nevertheless, if the BRW is quasi-transitive, then the global critical parameter depends only on the
 first-moment matrix (the same is true for $\mathcal F$-BRW). 
 Since the process in  Example \ref{ex:agelooptree} is not a quasi-transitive BRW (nor an $\mathcal F$-BRW), the trick is to find a transitive ageing BRW $(\mathbb T_d,\mathcal R)$, of which  $(\mathbb T_d,\mathcal R^*)$ 
 is a local
 modification. Then, on the one hand we can apply Theorem \ref{cor:maximality}, and compare $\lambda_s^*$ to $\lambda_w$;
 on the other hand we use the fact that the critical parameters of $(\mathbb T_d,\mathcal R)$
 coincide with those of the continuous-time BRW with the same first-moment matrix. We also stress that from
the expression of $\Phi^*(o,o|\lambda)$ in Equation \eqref{eq:phi*} and Equation \eqref{eq:lambdas1}, one could
derive (through easy, but tedious computation) the explicit expression of $\lambda_s^*$ (as in \cite[Example 1]{cf:BZmathematics},
see \cite{cf:tesiElena}).

The behaviour of the process in Example \ref{ex:agelooptree} can be seen from two different perspectives, according to the fact that we 
might use the process as a model for the spread of an endangered species or for disease/pest control.
In such models, we can view $\lambda$ as a fixed parameter which is a characteristic of the population.

In the case of endangered species, we see that encouraging local reproduction can lower the survival critical parameters, 
but ageing also plays a role. Indeed, the strength of the local reproduction rate, necessary to obtain $\lambda_w^*<\lambda$ or $\lambda_s^*<\lambda$, depends on the ageing parameters. Controlling the fertility decay with age, together with the
reproduction rates, becomes a reasonable strategy to prevent extinction.
On the other hand, in the case of pest/disease control, identifying sites where local reproduction drives the survival of the population, 
could be of great importance: tampering with the local reproduction rates and the fertility decay in these sites, may increase the
global and local survival critical parameters above the given parameter $\lambda$, driving this population to extinction.

In  Example \ref{ex:agelooptree}, we have seen that the addition of local reproduction, which takes into account the ageing parameter, in a single site, can destroy the pure global survival phase. In Example \ref{ex:homloops}, we see that in some cases, modifying just the ageing parameter, may eliminate the pure global survival phase.
\begin{Example}\label{ex:homloops}
Let $o$ be a fixed vertex in the homogeneous tree $\mathbb T_d$, and fix positive real numbers $k, k^*,\alpha, \alpha_0$.
Let $(\mathbb T_d,\mathcal R^*)$ be the ageing BRW, with the following reproduction rate functions
\[
r^*_{xy} = 
\begin{cases} 
	    {k^*}\cdot\exp(-\alpha_o t) & \text{if } x = y, \\
	      {k}\cdot\exp(-\alpha_o t) & \text{if } x = o, y \sim o, \\
   k\cdot\exp(-\alpha t)  & \text{if } x \neq o, y \sim x.
\end{cases}
\]
Let $\lambda_w^*$ and $\lambda_s^*$ be its global and local survival critical parameters, respectively.
Then, for any given $k$ and $k^*$, there are choices of $\alpha$ and $\alpha_o$ such that $\lambda_w^*<\lambda_s^*$
and other choices such that $\lambda_w^*=\lambda_s^*$.
In particular,  for any given $k$ and $k^*$, there are values of $\alpha$ such that, for some values of $\alpha_o$ we have
$\lambda_w^*<\lambda_s^*$
and for some other values of $\alpha_o$ we have $\lambda_w^*=\lambda_s^*$.
\end{Example}
\begin{proof}
Let us denote by $(\mathbb T_d,\mathcal R)$, the 
ageing BRW, with $\alpha=\alpha_o$. We proceed as in Example \ref{ex:agelooptree}.
The discrete-time counterpart of $(\mathbb T_d,\mathcal R)$ shares the first-moment matrix with
 the discrete-time counterpart of the continuous-time BRW $(\mathbb T_d,K)$,
$K=(k_{xy})_{x,y\in\mathbb T_d}$, $k_{xy}=k/(\alpha+1)$ if $x\sim y$, $k_{xx}=k^*$.
Since  $(\mathbb T_d,\mathcal R)$ and $(\mathbb T_d,K)$ are transitive BRWs,
 their critical parameters depend only on the first-moment matrix and share the global and local critical parameters.
 The  critical parameter for global survival coincides with the one of a branching process with an expected number of children equal to
 $(kd+k^*)/(\alpha+1)$, thus $\lambda_w = (\alpha +1)/(kd+k^*)$.
To compute $\lambda_s$, we write the first-return generating function $ \Phi(o,o|\lambda)$ of  $(\mathbb T_d,K)$, as
\begin{equation*}
	\begin{split}
\Phi(o,o|\lambda) & =
 \frac d{2(d-1)} \left(1-\frac{\lambda k^*}{\alpha+1}-\sqrt{\left(1-\frac{\lambda k ^*}{\alpha+1}\right)^2-\frac{4(d-1)\lambda^2 k^2}{(\alpha+1)^2}}\right)	\\
%
		  & =
		  \frac d{2(d-1)} \left(1-\sqrt{\left(1-\frac{\lambda k ^*}{\alpha+1}\right)^2-\frac{4(d-1)\lambda^2 k^2}{(\alpha+1)^2}}\right)
		  +\frac{(d-2)\lambda k^* }{2(d-1)(\alpha+1)}\\
		  \end{split}
\end{equation*}
Using the characterisation in Equation \ref{eq:lambdas2}, we obtain $\lambda_s = (\alpha+1)/ (k^* +2k \sqrt{d-1})$ which is strictly larger
than $\lambda_w$, for any choice of the parameters.

Now, we observe that, if $\alpha_o\neq\alpha$, then $(\mathbb T_d,\mathcal R^*)$ is a local modification of $(\mathbb T_d,\mathcal R)$.
We may compute $\lambda_s^*$, using the first-return generating function $\Phi^*$ of $(\mathbb T_d,\mathcal R^*)$.
Indeed $\Phi^*(o,o|\lambda)$ has the following expression:
\[
\frac{\alpha +1}{\alpha_o +1} \cdot \left(  \frac d{2(d-1)}\left(1-
\sqrt{\left(1-\frac{\lambda k^*}{(\alpha+1)}\right)^2 -\frac{ 4(d-1)(\lambda k)^2}{(\alpha+1)^2}}\right)
+ \frac{d-2}{2(d-1)} \frac{\lambda k^*}{\alpha+1}
 \right)
\]
If $\Phi^*(o,o|\lambda_w)>1$, then by Theorem \ref{cor:maximality}, $\lambda^*_s=\lambda^*_w<\lambda_w= (\alpha +1)/(kd+k^*)$.
We note that $\Phi^*(o,o|\lambda_w)$ is equal to
\[
\frac{\alpha +1}{\alpha_o +1} \cdot \left(  
\frac{d-d\sqrt{(1-\lambda k^*/(kd+k^*)^2-4(d-1) (\lambda k /(kd+k^*)^2}}{2(d-1)}+ \frac{d-2}{2(d-1)} \frac{\lambda k^*}{kd+k^*}
 \right)
\]
We know that for all $\alpha=\alpha_o$, we get the process $(\mathbb T_d,\mathcal R)$, which has a pure global survival phase.
If $\alpha_o=0$, then for every sufficiently large $\alpha$, the quantity $\Phi^*(o,o|\lambda_w)$ is larger than 1,
and there is no pure global survival phase.
On the other hand, once we have chosen such a large $\alpha$, there exist sufficiently small values $\alpha_o>0$ such that the above quantity is still larger than 1, and thus $\lambda_w^*=\lambda_s^*$.
This concludes the proof.
\end{proof}

In Examples \ref{ex:agelooptree} and  \ref{ex:homloops}, we see that we can take transitive BRWs, that have a pure global survival phase,
and perform a local modification (of the reproduction and/or of the ageing parameters), that lowers the critical parameters and removes the pure global survival phase.
It is natural to wonder if it is possible to take  a transitive BRW, which does not have the pure global survival phase,
and  perform local modifications on the reproduction and ageing parameters, to destroy this phase.
The answer is negative, as Theorem \ref{th:trans-mod} and Corollary \ref{cor:modquasitr} show.

\begin{Theorem}\label{th:trans-mod}
An irreducible ageing BRW with a pure global survival phase
cannot be a local modification of an irreducible ageing BRW with no pure global survival phase
and such that there is strong survival, for all values of $\lambda$, for which there is local survival with positive probability.
\end{Theorem}

\begin{proof}
	Let $(X,\mathcal R)$ be a BRW with a pure global phase (i.e. $\lambda_w < \lambda_s $) and let $(X,\mathcal R^*)$ be one of
	  its local modifications, with critical parameters $\lambda_w^*=\lambda_s^*$. Let $A \subseteq X$ be the finite set such that $r_{xy} \equiv r^*_{xy}$ $\forall x \notin A, y \in X.$ 
	  Consider now $\lambda \in (\lambda_w, \lambda_s)$ and suppose that $\lambda^*_w = \lambda^*_s$. By 
	  Theorem \ref{cor:maximality},
	  we have $\lambda^*_w \leq \lambda_w$.
      Denote by $\mathbf{q}(\cdot,\cdot)$ and by $\mathbf{q}^*(\cdot,\cdot)$ the extinction probabilities of $(X,\mathcal R)$
	  and $(X,\mathcal R^*)$, respectively, and recall that these quantities depend on the choice of $\lambda$.
	  Fix $\lambda\in (\lambda_w,\lambda_s)$.
	   By the definition of critical parameters, for that choice of $\lambda$, $\mathbf{q}(x,X)<\mathbf{q}(x,x)=1$ holds for every $x\in X$. Since 	   $(X,\mathcal R)$ is irreducible, we have $\mathbf{q}(x,x)=\mathbf{q}(x,C)$, for all $x \in X$, and for all non-empty, finite $C\subseteq X$.
	   In particular, we can  take $C=A$ and note that $\mathbf{q}(x,x)=\mathbf{q}(x,A)$, for all $x\in X$.

   According to a generalisation of \cite[Theorem 3]{cf:BZmathematics}
   (see \cite{cf:BZ-critical26}),
	   from $ \mathbf{q}(x,X) < \mathbf{q}(x,A)$, we deduce that $\mathbf{q}^*(x,X)<\mathbf{q}^*(x,A)= \mathbf{q}^*(x,x)$ for all $x \in X$.
    Since  $\lambda > \lambda_w \geq \lambda^*_s$,
    $(X,\mathcal R^*)$ survives locally, with positive probability. In other words, $\mathbf{q}^*(x,x)<1$, for all $x\in X$. This implies that 
    for our choice of $\lambda$, $(X,\mathcal R^*)$ there is local survival, with positive probability, but no strong survival
    (that is, local and global survival have both positive probability, but these probabilities are different).
    This  completes the proof.
       \end{proof}
       
It is easy to prove that Theorem \ref{th:trans-mod} implies the following corollary.
\begin{Corollary}\label{cor:modquasitr}
An irreducible, ageing BRW with a pure global survival phase 
cannot be a local modification of an irreducible, quasi-transitive, ageing BRW, with no pure global survival phase.
\end{Corollary}

\begin{proof}
		By \cite[Corollary 3.2]{cf:BZ14-SLS}, we know that for an irreducible quasi-transitive BRW, either  $\mathbf{q}(x,x)=1$
		for all $x \in X$, or 
		$\mathbf{q}(x,x)= \mathbf{q}(x,X) < 1$
		 for all $x \in X$. 
		 This means that for all values of $\lambda$ such that there is local survival with positive probability, there is strong 
		 survival, and by Theorem \ref{th:trans-mod} we know that this BRW cannot be a local modification of a BRW with pure global survival phase.
	    \end{proof}
We summarize the main results of this section in Table \ref{tb:modif}.

\begin{table}[H]
\captionof{table}{The second BRW is a local modification of the first one. 
We call a site with a self-reproduction loop a \textit{hub}.
The parameters of the first process, 
$\lambda_w$ and $\lambda_s$, are compared with those of the second process, $\lambda_w^*$ and $\lambda_s^*$.
Smaller parameter values make survival more likely, whereas larger values increase the probability of extinction.
We provide instructions for the local modification and indicate where examples and proofs can be found. }\label{tb:modif}
\tabcolsep=0.3cm 
        \begin{tabular}{ ccc|c }
\hline
\hline
 \rowcolor{lightgray!50}
\textbf{1st process} & \textbf{2nd process} & \textbf{Instructions} & \textbf{Where } \\  \hline
transitive &  not transitive & add one hub  &\\
no hubs &   
 one hub & reduce ageing at hub &Example \ref{ex:agelooptree}
 \\ 
$\lambda_w<\lambda_s$ &  $\lambda_w^*=\lambda_s^*\le\lambda_w$ &  &\\
 \hline
not transitive &  transitive & reduce reprod. at hub  &\\
one hub &   
 no hubs & increase ageing at hub &Example \ref{ex:agelooptree}
 \\ 
$\lambda_w=\lambda_s$ &  $\lambda_w\le\lambda_w^*<\lambda_s^*$ &  &\\
 \hline
 transitive &  not transitive & reduce ageing  &\\
all sites are hubs &   
all sites are hubs & at one site &Example \ref{ex:homloops}
 \\ 
$\lambda_w<\lambda_s$ &  $\lambda_w^*=\lambda_s^*\le\lambda_w$ &  &\\
 \hline
not transitive &  transitive & increase ageing   &\\
all sites are hubs &   
 all sites are hubs  & at one site &Example \ref{ex:homloops}
 \\ 
$\lambda_w=\lambda_s$ &  $\lambda_w\le\lambda_w^*<\lambda_s^*$ &  &\\
 \hline
  transitive &  transitive or not &  	IMPOSSIBLE   & Corollary \ref{cor:modquasitr}\\
 $\lambda_w=\lambda_s$ &  $\lambda_w^*<\lambda_s^*$ &  &\\
 \hline
\hline
\end{tabular}

\end{table}


\section{The expected number of individuals in an ageing BRW}
\label{sec:expectation}
In the previous section, we have studied the critical parameters
of ageing BRWs, subjected to local modifications. 
We have exploited the fact that the long-term behaviour of BRWs coincides 
with the one of their
discrete-time counterparts and, in many cases, this 
behaviour depends only on the first moments of the offspring distributions.
In particular, in those cases, the ageing BRW survives with positive probability
(or goes extinct) if and only if the continuous-time BRW, with the same expectation
of the offspring distribution, does.

Nevertheless, we know that the laws of an ageing BRW and of a continuous-time (classical) BRW, are different. In this section, we investigate the expected number of individuals, alive in an ageing BRW and in a continuous-time BRW, as a function of time.
We recall that this expected value describes, with large probability, the behaviour of the system with a sufficiently large number
of initial individuals, thanks to Kurtz's limit theorems \cite{cf:Kurtz1, cf:Kurtz2}.
To keep things simple, we consider branching processes, rather than BRWs.

Let $S_t$ denote the expected number of individuals that are alive at time $t\ge0$, for a classical continuous-time branching process, 
with rate $\lambda$. It is well-known that $S_t$ satisfies the
differential equation $  \dot{S_t} = (\lambda -1) S_t $, where by $\dot{S_t}$
we mean the derivative of $S_t$, with respect to $t$
 (see, for example, \cite[Chapter III]{cf:AthNey}).
 Its solution is 
 $S_t = S_0 \cdot e^{(\lambda -1)t}$, 
 where $S_0$ is the initial number of individuals in the population.

Consider now the ageing branching process  with   $r (t)=e^{-\alpha t}$,
 for all $t\ge0$, $\alpha > 0$.
This means that the ability of each individual to reproduce decreases exponentially with time. 
The following result gives the expression of the expected number of individuals in this ageing branching process.
\begin{Theorem}\label{th:ode}
Fix $\alpha > 0$.
 Let  $(X,\mathcal R)$ be the ageing branching process, where $X=\{x\}$ and  $r (t)=e^{-\alpha t}$,
 for all $t\ge0$. Let $V_t$ be the expected number of individuals that are alive at time $t\ge0$.
 Then, for all fixed $\lambda>0$, $V_t$ is as follows.
 \begin{enumerate}
 \item
 If $\lambda\neq\alpha$,
$   V_t= V_0\left(\frac{\lambda}{\lambda-\alpha} e^{(\lambda-\alpha-1)t}-\frac{\alpha}{\lambda-\alpha} e^{-t}\right)$.   
\item
If  $\lambda=\alpha$, then $ V_t= V_0(1+\lambda t) e^{-t}$.
\end{enumerate}
\end{Theorem}
\begin{proof}
 Recalling that each individual has an exponentially distributed lifetime of parameter 1, the expected number of offspring,
produced up to time $t$ by an individual born at time 0, is given by
\[
\int_{0}^{t} e^{-x} \int_{0}^{x} {\lambda e^{-\alpha s}\,ds} + 
\int_{t}^{+\infty} {e^{-x}} \int_{0}^{t} {\lambda e^{-\alpha s}\,ds}.
\]

Let $V_t$ denote the expected number of individuals alive at time $t \geq 0$, and let $N_t$ be the total
number of individuals born up to time $t$. Clearly $N_0 = V_0$, since at time $0$ the only individuals present are the initial ones.  
Moreover, $\dot{N}_s\,ds$ is the expected number of births in the infinitesimal time interval $ds$.

The number $V_t$
 can be decomposed into the number of initial individuals, that are still alive, plus the number of individuals that were born at time $s>0$ and
 are still alive. 
Each of the initial individuals $V_0$ has an exponential lifetime with parameter $1$, hence survives until time $t$ with probability 
$e^{-t}$. 
This yields a contribution 
$
V_0\, e^{-t}$.
On the other hand, consider an individual born at time $s \in [0,t]$. 
They survive at least until time $t$ with probability $e^{-(t-s)}$. 
The expected number of individuals born in $(0,t)$, still alive
at time $t$, is equal to
$
\int_0^t \dot{N}_s\, e^{-(t-s)} \, ds
$.
Adding these two contributions, we obtain 
\begin{equation}
\label{eq:Vt}
	V_t = V_0 \, e^{-t} + \int_0^t \dot{N}_s\, e^{-(t-s)} \, ds.
\end{equation}
We compute $\dot{N}_t$, separating the contribution of 
individuals alive at time 0, from the one of individuals born at time $s>0$: 
\begin{equation}\label{eq:N.t}
	\dot{N}_t = \lambda V_0\, e^{-(\alpha+1)t} 
    + \lambda\int_{0}^{t} { \dot{N}_s\,  e^{-(\alpha+1)(t-s)}\, ds}.
\end{equation}
We note that in the first summand we take into account
that the initial population $N_0 = V_0$ reproduces with rate $\lambda e^{-\alpha t}$, provided it has survived up to time $t$ (which occurs with probability $e^{-t}$). 
The second summand represents the contribution of all individuals born at times $s\in (0,t)$: 
each individual reproduces, at time $t$, with a rate $\lambda e^{-\alpha (t-s)}$, provided that it has survived up to an age $(t-s)$, 
an event that occurs with probability $e^{-(t-s)}$.  

\noindent Now, we differentiate the two members of Equation \eqref{eq:Vt}, with respect to $t$.
\begin{equation}\label{eq:V.t}
   \begin{split}
	\dot{V_t}
	&= \frac{d}{dt}\big(V_0\, e^{-t}\big)
	+ \frac{d}{dt}\left(\int_0^t \dot{N}_s\,e^{-(t-s)}\,ds\right)\\
	&= -V_0\, e^{-t}
	+ \dot{N}_t
	+ \int_0^t \dot{N}_s\,\frac{d}{d t}\big(e^{-(t-s)}\big)\,ds\\
	&= -V_0\, e^{-t}
	+ \dot{N}_t
	- \int_0^t \dot{N}_s\,e^{-(t-s)}\,ds.
\end{split} 
\end{equation}
From Equations \eqref{eq:Vt} and \eqref{eq:V.t}, we get
$\dot{V}_t = \dot{N}_t - V_t$ and
\begin{equation}\label{eq:N.1}
\dot{N}_t = \dot{V}_t + V_t.
 \end{equation}
Therefore, Equation \eqref{eq:N.t} can be rewritten as
\begin{equation}\label{eq:N.2}
 \dot{N}_t = \lambda V_0 e^{-(\alpha+1)t} + \lambda
 \int_{0}^{t} (\dot{V}_s + V_s) \,  \, e^{-(\alpha+1)(t-s)} \, ds.   
\end{equation}
If we differentiate Equation \eqref{eq:N.1}, we get
\begin{equation}\label{eq:N2.1}
  \ddot{N}_t = \ddot{V}_t +\dot{V}_t; 
\end{equation}
while differentiating Equation \eqref{eq:N.2} gives
\begin{equation}\label{eq:N2.2}
\begin{split}
   \ddot{N}_t &=
-(\alpha+1) \lambda V_0 e^{-(\alpha+1)t} 
	+ \lambda (\dot{V}_t + V_t)  
	- \lambda \int_0^t (\dot{V}_s + V_s) \, (\alpha+1) e^{-(\alpha+1)(t-s)} \, ds\\
    &
    =-(\alpha+1)\dot{N}_t+ \lambda(\dot{V}_t + V_t)
    =(\lambda - \alpha - 1)(\dot{V}_t + V_t).
\end{split}
\end{equation}
From Equation \eqref{eq:N2.1}, we have $\ddot{V}_t = \ddot{N}_t -\dot{V}_t$.
Together with Equation \eqref{eq:N2.2}, this yields the following differential 
equations (where $V_0$ is the number of initial individuals):
\begin{equation}\label{eq:ODE}
	\begin{cases}
		\ddot{V}_t = \dot{V}_t(\lambda -\alpha-2) +V_t(\lambda -\alpha-1);\\
		\dot{V}_0 = (\lambda -1)V_0.
	\end{cases}
\end{equation}
It is easy to check that the two cases in the statement of Theorem \ref{th:ode} are the solutions of 
\eqref{eq:ODE}.
\end{proof}
We want to compare the function $V_t$ of the ageing branching process in Theorem \ref{th:ode}
with the function $S_t$ of a continuous-time branching process that has
the same first moment. To this aim, in the following definition, we define a branching process
\textit{equivalent} to $(X,\mathcal R)$.
The term is justified by the facts that the two processes share the same 
destiny (they both survive with positive probability or they both go almost
surely extinct) and that they have the same expected number of offspring.
\begin{Definition}
    Fix $\alpha, \lambda>0$ and let $(X,\mathcal R)$
    be the ageing branching process defined in Theorem \ref{th:ode}.
    The \textsl{equivalent} branching process is the continuous-time branching process, where individuals breed at the arrival times of a
    Poisson process of rate $\lambda/(1+\alpha)$ and die at rate 1. 
\end{Definition}
We are now ready to state the following result, which compares the expectations of the number of individuals in these two processes.
The proof is a simple consequence of the expressions
$S_t=S_0\cdot e^{(\lambda /(\alpha+1)-1)t}$ and Theorem \ref{th:ode} for $V_t$.
\begin{Theorem}\label{th:compare}
  Fix $\alpha, \lambda>0$ and let  $(X,\mathcal R)$ be the ageing branching process, where  $r(t)=e^{-\alpha t}$,
 for all $t\ge0$ ($X=\{x\}$). Let $(X,\lambda/(1+\alpha))$ be its equivalent branching process. Denote by $V_t$ and 
 $S_t$ the expected number of individuals who are alive at time $t\ge0$, in $(X,\mathcal R)$ and in $(X,\lambda/(1+\alpha))$,
 respectively. Suppose that $V_0=S_0$. As $t\to\infty$, the possible cases are the following.
 \begin{enumerate}
     \item If $\lambda>\alpha+1$, then both $V_t$ and $S_t$ tend to infinity and
     $V_t\sim\frac\lambda{\lambda-\alpha}V_0 e^{(\lambda-\alpha-1)t}$ is eventually larger than $S_t$.
     \item If $\lambda=\alpha+1$, then $S_t=V_0$, for all $t\ge0$, and $V_t\to V_0\lambda $.
     \item If $\alpha<\lambda<\alpha+1$, then both $V_t$ and $S_t$ tend to 0,  $V_t\sim\frac\lambda{\lambda-\alpha}V_0 e^{(\lambda-\alpha-1)t}$
      and $S_t$ is eventually larger than $V_t$.
      \item If $\lambda=\alpha$, then both $V_t$ and $S_t$ tend to 0,  $V_t\sim\alpha t V_0 e^{-t}$
      and $S_t$ is eventually larger than $V_t$.
       \item If $\lambda<\alpha$,  then both $V_t$ and $S_t$ tend to 0,  $V_t\sim\frac\alpha{\alpha-\lambda}V_0 e^{-t}$
      and $S_t$ is eventually larger than $V_t$.
 \end{enumerate}
\end{Theorem}
We note that, while the qualitative asymptotic behaviour of $V_t$ and $S_t$
are the same (diverging, converging to a finite positive limit or to 0), the
quantitative behaviour differs. Moreover, if $\lambda\in(1,\alpha+1)$,
$S_t$ decreases for all $t\ge0$ and $V_t$ increases for small values of $t$.
As an exemplification of the possible cases, see Figure \ref{fig:quattro_grafici}. We note that cases 4. and 5. are quite similar,
therefore the last one is not displayed in this figure.
It is worth noting that, whenever $\lambda>1$, for small values of $t$, $V_t$ is increasing (this is due to the initial condition in
Equation \eqref{eq:ODE}). Even if here we observe this phenomenon for the case of ageing  branching processes,
we mention that it can be observed also for ageing BRWs with more than one site (see \cite{cf:tesiElena}).



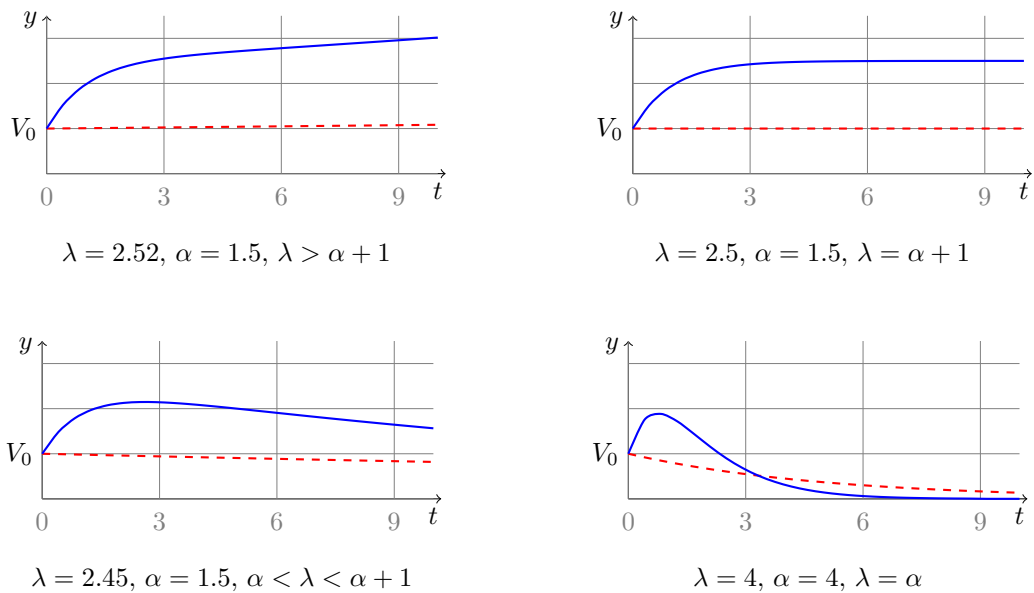
\begin{figure}[H]
	\centering
	\begin{minipage}{0.45\textwidth}
		\centering
		\begin{tikzpicture}[xscale=0.52,yscale=0.6]
			\draw[->] (0,0) -- (10.2,0);
			\draw[->] (0,0) -- (0,3.5);
			
			\node[left] at (0,1) {$V_0$};
			\node[left] at (0,3.4) {$y$};
			\node[below] at (10,0) {$t$};
			
			\foreach \i in {0,3,6,9} {
				\draw [very thin,gray] (\i,0) -- (\i,3.5) node [below=2pt] at (\i,0) {$\i$};
			}
			\foreach \i in {0,1,2,3} {
				\draw [very thin,gray] (0,\i) -- (10,\i);
			}
			
			\draw[domain=0:10, dashed, red, thick] plot(\x,{exp((2.52/2.5-1)*\x)});
			
			\draw[domain=0:10, blue, thick, smooth] plot(\x,{
				2.52*exp((2.52-2.5)*\x)/(2.52-1.5)-1.5*exp(-\x)/(2.52-1.5)
			});
		\end{tikzpicture}
		
		\vspace{0.2cm}
		\centering
		{$\lambda=2.52$, $\alpha=1.5$, $\lambda>\alpha+1$}
	\end{minipage}
	\hspace{0.7cm}
	\begin{minipage}{0.45\textwidth}
		\centering
		\begin{tikzpicture}[xscale=0.52,yscale=0.6]
			\draw[->] (0,0) -- (10.2,0);
			\draw[->] (0,0) -- (0,3.5);
			
			\node[left] at (0,1) {$V_0$};
			\node[left] at (0,3.4) {$y$};
			\node[below] at (10,0) {$t$};
			
			\foreach \i in {0,3,6,9} {
				\draw [very thin,gray] (\i,0) -- (\i,3.5) node [below=2pt] at (\i,0) {$\i$};
			}
			\foreach \i in {0,1,2,3} {
				\draw [very thin,gray] (0,\i) -- (10,\i);
			}
			
			\draw[domain=0:10, dashed, red, thick] plot(\x,1);
			
			\draw[domain=0:10, blue, thick, smooth] plot(\x,{
				2.5-1.5*exp(-\x)
			});
		\end{tikzpicture}
		
		\vspace{0.2cm}
		\centering
		{$\lambda=2.5$, $\alpha=1.5$, $\lambda=\alpha+1$}
	\end{minipage}
	
	\vspace{0.8cm}
	
	\begin{minipage}{0.45\textwidth}
		\centering
		\begin{tikzpicture}[xscale=0.52,yscale=0.6]
			\draw[->] (0,0) -- (10.2,0);
			\draw[->] (0,0) -- (0,3.5);
			
			\node[left] at (0,1) {$V_0$};
			\node[left] at (0,3.4) {$y$};
			\node[below] at (10,0) {$t$};
			
			\foreach \i in {0,3,6,9} {
				\draw [very thin,gray] (\i,0) -- (\i,3.5) node [below=2pt] at (\i,0) {$\i$};
			}
			\foreach \i in {0,1,2,3} {
				\draw [very thin,gray] (0,\i) -- (10,\i);
			}
			
			\draw[domain=0:10, dashed, red, thick] plot(\x,{exp((2.45/2.5-1)*\x)});
			
			\draw[domain=0:10, blue, thick, smooth] plot(\x,{
				2.45*exp((2.45-2.5)*\x)/(2.45-1.5)-1.5*exp(-\x)/(2.45-1.5)
			});
		\end{tikzpicture}
		
		\vspace{0.2cm}
		\centering
		{$\lambda=2.45$, $\alpha=1.5$, $\alpha<\lambda<\alpha+1$}
	\end{minipage}
	\hspace{0.7cm}
	\begin{minipage}{0.45\textwidth}
		\centering
		\begin{tikzpicture}[xscale=0.52,yscale=0.6]
			\draw[->] (0,0) -- (10.2,0);
			\draw[->] (0,0) -- (0,3.5);
			
			\node[left] at (0,1) {$V_0$};
			\node[left] at (0,3.4) {$y$};
			\node[below] at (10,0) {$t$};
			
			\foreach \i in {0,3,6,9} {
				\draw [very thin,gray] (\i,0) -- (\i,3.5) node [below=2pt] at (\i,0) {$\i$};
			}
			\foreach \i in {0,1,2,3} {
				\draw [very thin,gray] (0,\i) -- (10,\i);
			}
			
			\draw[domain=0:10, dashed, red, thick] plot(\x,{exp((-1/5)*\x)});
			
			\draw[domain=0:10, blue, thick, smooth] plot(\x,{
				(1+4*\x)*exp(-\x)
			});
		\end{tikzpicture}
		
		\vspace{0.2cm}
		\centering
		 {$\lambda=4$, $\alpha=4$, $\lambda=\alpha$}
	\end{minipage}
	\vspace{0.3cm}
	\caption{The expected value of individuals of the ageing branching process with breeding intensity $\lambda e^{-\alpha t}$ (blue, solid line) and of the continuous-time branching process with intensity $\lambda/(1+\alpha)$ (red, dashed line), as functions of the time variable $t$.}
	\label{fig:quattro_grafici}
\end{figure}

\section{Conclusions}
\label{sec:concl}

Branching random walks
 are  models for the number of individuals in a population that breeds and dies,
and for the number of infected individuals during an epidemic outbreak.
In the classical case, the behaviour of the process depends on its structure (the reproduction rates across sites)
and on the intensity parameter $\lambda$. 
In this study we introduce an ageing function, which takes into account the natural decline of fertility of individuals, as they age.
In contagious diseases, this decline represents the fact that transmissibility is at its peak when the individual has been recently infected, and decreases
over time. In particular, we focus on the case of exponential decay of the fertility/transmissibility rate, parametrised by the parameter 
$\alpha$.

In real situations one aims either at survival (protecting endangered species, for example) or at extinction (epidemic or pest control).
Increasing or decreasing the intensity $\lambda$ would seem to be the most natural option, but it might be impractical, since that would
require a global change of the characteristics of the population. It is therefore important to understand if and how local changes affect survival.
Moreover, when the model space is infinite, it is possible to observe local extinction (the population disappears from any finite set),
while there is global survival (the population persists as a whole). 
 Thus, the observation of the process on a finite set
provides exhaustive information on the asymptotic behaviour only if there is no pure global survival phase, that is, when
$\lambda_w=\lambda_s$. Indeed, when $\lambda_w<\lambda<\lambda_s$, one might think that the population is destined to extinction,
which is true from a local point of view, but local changes may provide  local survival, for the same $\lambda$ (see Table \ref{tb:modif}).

For endangered species, given a $\lambda$  such that there is local or global extinction,
 increasing the local reproduction rate and reducing the fertility decay rate, at one single location, can put
 the populaton in the range of local survival. On the other hand, if the population is sustained by a local environment where
 reproduction is fast and ageing slow, disrupting this ecosystem might lead to extinction.
 
 On the other hand, in disease (or pest) control, we may have an infectious disease that is destined to disappear, but it could turn into an epidemic outbreak if a community increases the frequency of contacts
 (higher reproduction rates) and/or lowers the speed of decay of transmissibility. On the contrary, a disease which persists thanks to
 some hubs, that is, sites with faster reproduction rates and/or slower ageing, can disappear once we isolate and treat those hubs.
 
 The importance of hubs, or hot spots, has been pointed out by earlier works on mathematical epidemic modelling, both from a
 theoretical point of view,  \cite{cf:Barabasi, cf:volz08} and with numerical simulations, \cite{cf:buzna06, cf:okabe21}.
   Our study confirms this importance: the treatment or creation of hubs can change the ultimate fate of a population or of an epidemic.
 
 In classical branching processes, the expected number $V_t$ of individuals alive at time $t$, can show only three behaviors:
 it increases for all $t\ge0$ (supercritical case), it remains constant (critical case), or it decreases for all $t\ge0$ (subcritical case).
 The introduction of an ageing parameter enriches the landscape of possibilities. We can compare the behaviour of $V_t$ with
 that of the expected number $S_t$ of individuals in a process with the same expected number of offspring, but with no ageing. 
 In the supercritical cases $V_t$ increases for all $t\ge0$, tends to infinity and grows faster than $S_t$; in the critical case $V_t$ increases and tends to a constant limit.
 In the subcritical case, $V_t$ tends to 0 and is eventually smaller than $S_t$, but there are cases (namely when $\lambda>1$) where
 $V_t$ exhibits an initial increase over time.
This phenomenon, absent if we do not add ageing in the model, underlines the importance of observation 
of populations for a sufficiently large amount of time: a temporary increase of the population is possible, even if
the ultimate fate is extinction.

Future work could focus on the short-term behaviour of ageing processes with multiple sites, and on both short and long-term behaviour
 of ageing BRW with constraints on the number of individuals in some/all sites. This could be investigated both through simulations
 and the theoretical study of models with a small number of sites.


\begin{thebibliography}{999}


%
%
%

%

%
%
%
%
%

%
%
%
%
%
%
%
%




%
%

\bibitem{cf:AthNey}
Athreya, K.B.; Ney, P.E.
\textit{Branching Processes}.
Die Grundlehren der mathematischen Wissenschaften, 196;
Springer: {Berlin/Heidelberg, Germany}, 1972.

\bibitem{cf:ath-73}
      Athreya, K.; Ney, P.
Limit theorems for the means of branching random walks.
In Transactions of the Sixth Prague Conference on Information Theory, 
Statistical Decision Functions, Random Processes, Prague, 1973,  63--72.

\bibitem{cf:asm76-1}
Asmussen, S.; Kaplan, N.
Branching random walks I.
{\em Stochastic Process. Appl.} {\bf 1976}, {\em 4}, 1--13.

%

 
\bibitem{cf:BZ2}
Bertacchi, D.; Zucca, F.
Characterization of the critical values of branching random walks on
weighted graphs through infinite-type branching processes.
{\em J.~Stat.~Phys.}  {\bf 2009}, {\em{134}}, 53--65.


\bibitem{cf:BZ14-SLS}
Bertacchi, D.; Zucca, F.
Strong local survival of branching random walks is not monotone.
{\em Adv.~Appl.~Probab.}  {\bf 2014}, {\em{46}},  400--421.


\bibitem{cf:BZ17}
Bertacchi, D.; Zucca, F.
A generating function approach to branching random walks.
{\em Braz.~J.~Probab.~Stat.}  {\bf 2017}, {\em{31}}, 229--253.

\bibitem{cf:BZmathematics}
Bertacchi, D.; Zucca, F.
On the critical parameters of branching random walks.
{\em Mathematics}  {\bf 2025}, {\em{13}}, 2962.

\bibitem{cf:BZ-critical26}
Bertacchi, D.; Zucca, F.
Critical parameters of germ-monotone families of branching random walks.
{\em Preprint}  {\bf 2026}, arXiv:2602.21062 .


 \bibitem{cf:buzna06}
Buzna L., Peters K., Helbing D.
Modelling the dynamics of disaster spreading in networks.
{\em Physica A: Statistical Mechanics and its Applications}
{\bf 2006}, {\em 363}, 132--140.

\bibitem{cf:daipra2}
Cerf R., Dai Pra P., Formentin M., Tovazzi D.
Rhythmic behavior of an Ising Model with dissipation at low temperature.
{\em ALEA} {\bf 2021}, {\em 18}, 439--467.

\bibitem{cf:Clancy98}
Clancy, D.; O'Neill, P.
Approximation of epidemics by inhomogeneous birth-and-death processes.
{\em Stochastic Process. Appl.}  {\bf 1998}, {\em 73}, 233--245.

\bibitem{cf:daipra1}
Collet F., Dai Pra P., Formentin M.,
Collective periodicity in mean-field models of cooperative behavior.
{\em Nonlinear Differential Equations Appl.} {\bf 2015}, {\em 22}, 1461--1482.

\bibitem{cf:daipra3}
Dai Pra P., Marini E.
Noise-induced oscillations for the mean-field dissipative contact process.
{\em El.~J.~Probab.} {\bf 2025}, {\em 30}, 98.

\bibitem{cf:Barabasi}
Dezso Z.; Barabási A.-L.
Halting viruses in scale-free networks.
{\em Phys. Rev. E} {\bf 2002}, {\em 65}, 1--4.

	
\bibitem{cf:feller59}
Feller, W.
The birth and death processes as diffusion processes.
{\em  J.~Math.~Pures Appl.} {\bf 1959} {\em 9}, 301--345.


\bibitem{cf:GW1875}
Galton, F.; Watson, H.W.
On the probability of the extinction of
families.
{\em J.~Anthropol.~Inst.~Great Br.~Irel.}
 {\bf 1875}, {\em{4}}, 138--144.
 
\bibitem{cf:Griff73}
Griffiths, D. A.
Multivariate birth-and-death processes as approximations to epidemic processes.
{\em J.~Appl.~Probab.}  {\bf 1973}, {\em{10}}, 15--26.

\bibitem{cf:Harris63}
Harris, T.E. 
\textit{The Theory of Branching Processes}; Springer: {Berlin/Heidelberg, Germany}, 1963.

\bibitem{cf:joffe73}
Joffe, A.; Moncayo, A. R.
Random variables, trees, and branching random walks.
{\em Advances in Math.} {\bf 1973}, {\em 10}, 401--416.

\bibitem{cf:asm76-3}
Kaplan, N.; Asmussen, S.
Branching random walks II.
{\em Stoch.~Proc.~Appl.} {\bf 1976}, {\em 4}, 15--31.

\bibitem{cf:kendall60}
Kendall, David G.
Birth-and-death processes, and the theory of carcinogenesis,
{\em Biometrika} {\bf 1960} {\em 47}, 13--21.


\bibitem{cf:Kurtz1}
Kurtz, T.G.
Solutions of ordinary differential equations as limits of pure jump Markov processes, 
{\em J.~Appl.~Probab.} {\bf 1970}, {\em 7}, 49--58.

\bibitem{cf:Kurtz2}
Kurtz, T.G., 1971.
Limit theorems for sequences of jump Markov processes approximating ordinary differential processes, 
{\em J.~Appl.~Probab.} {\bf 1971}, {\em 8}, 344--356.


\bibitem{cf:Ligg1}
Liggett, T.M.
{Branching random walks and contact processes on homogeneous trees}.
{\em Probab.~Theory Relat. Fields}  {\bf 1996}, {\em{106}}, 495--519.

\bibitem{cf:Ligg2}
Liggett, T.M.
{Branching random walks on finite trees}.
In {\em Perplexing problems in probability, Progr. Probab.}
{\bf{44}}, Birkh\"auser Boston, Boston, MA, 1999, 315-330.

\bibitem{cf:MachadoMenshikovPopov}
Machado, F.P.; Menshikov, M.V.; Popov, S.
Recurrence and transience of multitype branching random walks.
{\em Stoch.~Proc.~Appl.}  {\bf 2001}, {\em{91}},  21--37.



\bibitem{cf:MP03}
Machado, F.P.; Popov, S.
Branching random walk in random environment on trees.
{\em Stoch.~Proc.~Appl.}  {\bf 2003}, {\em{106}},  95--106.


\bibitem{cf:MadrasSchi}
Madras, N.; Schinazi, R. 
Branching random walks on trees.
{\em Stoch.~Proc.~Appl.}  {\bf 1992}, {\em{42}},  255--267.
 
 

\bibitem{cf:tesiElena}
E.~Montanaro, 
Stochastic models for infections.
PhD Thesis (in progress),  2026, Università La Sapienza, Roma.

  
  \bibitem{cf:Muller08-2}
M\"uller, S.
Recurrence for branching Markov chains.
{\em Electron.~Commun.~Probab.} {\bf 2008}, {\em 13}, 576--605.

 \bibitem{cf:okabe21}
Okabe, Y.; Shudo, A. 
Microscopic Numerical Simulations of Epidemic Models on Networks. 
{\em{Mathematics}} \textbf{2021}, {\em 9}, 932.


%
%

\bibitem{cf:PemStac1}
Pemantle, R.; Stacey, A.M.
The branching random walk and
contact process on Galton--Watson and nonhomogeneous trees.
{\em Ann.~Probab.}  {\bf 2001}, {\em{29}},  1563--1590.



\bibitem{cf:Stacey03}
Stacey, A.M.
Branching random walks on quasi-transitive graphs.
{\em Combin.~Probab.~Comput.} {\bf 2003}, {\em 12}, 345--358.


 \bibitem{cf:volz08}
Volz E.
SIR dynamics in random networks with heterogeneous connectivity.
(2008) 
{\em J.~Math.~Biology} {\bf 2008}, {\em 56}, 293--310.

    
\bibitem{cf:Woess09}
         Woess, W.
         \textit{Denumerable Markov Chains,
 Generating Functions, Boundary Theory, Random Walks on Trees};
 EMS Textbooks in Mathematics,
 European Mathematical Society (EMS): Z\"urich, Switzerland, 2009.

 


\end{thebibliography}
\end{document}